\documentclass[12pt]{article} 
\usepackage{amsmath, amssymb, amscd, amsthm, mathrsfs, hyperref}
\usepackage{amscd}
\usepackage{array}
\usepackage{amssymb, ifpdf, xspace, dsfont, color}
\usepackage{fancyhdr, enumitem}
\usepackage{graphicx, caption}

\usepackage[left=1.5in,top=1in,right=1in]{geometry} 

\newtheorem{thm}{Theorem}
\newtheorem{defn}{Definition}
\newtheorem{cor}[thm]{Corollary}
\newtheorem{lemma}[thm]{Lemma}
\newtheorem{prop}[thm]{Proposition}

\newtheoremstyle{example}{\topsep}{\topsep}%
 {}
 {}
 {\bfseries}
 {.}
 { }
 {}

\theoremstyle{example}

\newcommand{\R}{\mathbb{R}}

\newcommand{\Z}{\mathbb{Z}}

\renewcommand{\hat}[1]{\widehat{#1}}
\renewcommand{\vec}[1]{{#1}}

\newcommand{\abs}[1]{\left| {#1} \right|}
\renewcommand{\chi}{\mathds{1}}
\newcommand{\spn}{\textnormal{span}}
\newcommand{\cnt}{\textnormal{Count}}

\title{{A higher moment formula for the Siegel--Veech transform over quotients by Hecke triangle groups}}
\author{Samantha~K.~Fairchild}
\date{}

\numberwithin{equation}{section}
\numberwithin{thm}{section}

\begin{document}
\pagestyle{fancy}

\maketitle
\abstract{
We compute higher moments of the Siegel--Veech transform over quotients of $SL(2,\R)$ by the Hecke triangle groups. After fixing a normalization of the Haar measure on $SL(2,\R)$ we use geometric results and linear algebra to create explicit integration formulas which give information about densities of $k$-tuples of vectors in discrete subsets of $\R^2$ which arise as orbits of Hecke triangle groups. This generalizes work of W.~Schmidt on the variance of the Siegel transform over $SL(2,\R)/SL(2,\Z)$.} 


\section{Introduction}
The Siegel--Veech transform maps a function on $\R^2$ to a function on sets of translation surfaces. This powerful transformation gives information about the asymptotic density of saddle connections \cite{Veech98} and and cusp excursions \cite{A-M_Log_Laws_I}. On connected strata of translation surfaces, the Siegel--Veech transform is integrable \cite{Veech89} and in $L^2$ with respect to the Masur--Veech measure \cite{AthreyaCheungMasurL2}. In \cite{Veech89}, Veech also showed that the Siegel--Veech transform is integrable over closed $SL(2,\R)$ orbits of Veech surfaces with respect to the induced Haar measure. Building on work of Siegel, Schmidt, and Rogers \cite{Siegel45, Schmidt60, Rogers55} we compute higher moments of the Siegel--Veech transform over sets of surfaces with the Hecke triangle groups as their stabilizer group. 

The Hecke triangle group $H_q$ for integers $q\geq 3$ is the discrete subgroup of $SL(2,\R)$ generated by
		$$S = \begin{bmatrix}
			0 & -1 \\ 1 & 0
		\end{bmatrix} \quad \text{and} \quad T = \begin{bmatrix}
			1 & \lambda_q \\ 0 & 1
		\end{bmatrix}, \text{ where } \lambda_q = 2\cos\left(\frac{\pi}{q}\right).$$
Note $H_3 = SL(2,\Z)$ and for all $q\geq 3$, $H_q$ has finite co-volume in $SL(2,\R)$. For more information on Hecke triangle groups see \cite{Lang_Lang}. 

Let $V_q$ be the discrete subset of $\R^2$ defined by $$V_q = H_q\cdot \begin{bmatrix} 1 \\ 0 \end{bmatrix},$$ which corresponds to a subset of saddle connections of a translation surface when $q$ is odd (see section~\ref{sec:Hecke_Triangle_As_Veech}). Define $Y_q = SL(2,\R) / H_q$ with corresponding Haar probability measure $\mu$. Let $B_c((\R^2)^k)$ be the set of bounded measurable functions with compact support on $(\R^2)^k$.
\begin{defn}\label{def:S--VTransform}
For $f\in B_c((\R^2)^k)$ and, by abuse of notation $g = [g] \in Y_q$, we define the \textit{Siegel--Veech transform} by 
		$$\hat{f}(g) = \sum_{(v_1,\ldots, v_k)\in V_q^k} f(g(v_1,\ldots,v_k)).$$
\end{defn}
In the above definition $k=1$ is the classical Siegel--Veech transform, and for particualar $f \in B_c((\R^2)^k)$ of the form $f(x_1,\cdots, x_k) = h(x_1)\cdots h(x_k)$ for $h \in B_c(\R^2)$, $\hat{f}$ corresponds to the $k$th power of the classical Siegel--Veech Transform of $h$ on $\R^2$. Veech proved that the classical Siegel--Veech transform is integrable with the following formula from section 16 of \cite{Veech98}.
\begin{thm} \label{thm:Veech}
			For $f \in B_c(\R^2)$,
			$$\int_{Y_q} \hat{f}(g)\,d\mu(g) = \frac{1}{c(q)} \int_{\R^2} f(x) \,dx,$$
			where the Siegel--Veech constant is given by $$c(q) \stackrel{def}{=} \pi \left(\pi - \frac{\pi}{q} - \frac{\pi}{2}\right).$$ 
		\end{thm}

		We will first prove the following theorem which computes the square of the classical Siegel--Veech transform on $B_c(\R^2)$. To state the theorem, we introduce the following two definitions:
		\begin{defn}[Set of non-vanishing determinants]
			Let			
			\begin{align} \label{eq:N_q}
				N_q &\stackrel{def}{=} \{n \in \Z[\lambda_q]\setminus\{0\}: \text{ there exists } \vec{v_1}, \vec{v_2} \in V_q \text{ with } \det(\vec{v_1}\,\vec{v_2}) = n\}\nonumber\\
		&= \left\{n\in \Z[\lambda_q]\setminus\{0\} : \text{ there exists } 0\leq m < \lambda_q |n| \text{ with } \begin{bmatrix}
			m \\ n
		\end{bmatrix} \in V_q\right\},
		\end{align}
		\end{defn}
		where Equation~\ref{eq:N_q} will be proved in Lemma~\ref{lem:N_q equality}. 
		\begin{defn}[$q$-geometric Euler totient function]
			For $b \in \Z[\lambda_q]$ define
			$$\varphi_q(b) = \#\left\{1\leq a\leq \lambda_q\abs{b}: \begin{bmatrix} a\\ b \end{bmatrix} \in V_q\right\}.$$
		\end{defn}	
		Note that $\lambda_3 = 1$ and $V_3 = SL(2,\Z)\cdot e_1$ so $\varphi_3$ reduces to the standard Euler totient function.
		
\begin{thm} \label{thm:The_Big_Kahuna}
		Let $f \in B_c(\R^2\times \R^2)$, $N_q$ be the set of non-vanishing determinants, and $\varphi_q$ the \textit{$q$-geometric Euler totient function}.
		Then,
		\begin{equation}
		\int_{Y_q} \hat{f} \,d\mu(g)= \sum_{n\in N_q} \frac{\varphi_q(n)}{c(q)}\int_{SL(2,\R)} f\left(g J_n\right) \,d\eta + \frac{1}{c(q)} \int_{\R^2} \left( f(x, -x) + f(x,x) \right)\,dx
		\label{eq:Full_Hecke_Formula}
		\end{equation}
		where $J_n = \begin{bmatrix} 1 & 1\\ 0 & n \end{bmatrix}$, $\mu$ is the Haar probability measure on $Y_q$, $\eta$ is Haar measure on $SL(2,\R)$ normalized so $\eta(Y_3) = \frac{\pi^2}{6}$, and $dx$ is the Lebesgue measure on $\R^2$ normalized so the area of the unit square is 1. 
	\end{thm}	
	
	Note $\hat{f}$ is uniformly bounded by Lemma~16.10 of \cite{Veech98}, so both sides of Equation~\ref{eq:Full_Hecke_Formula} are finite.  The proof of Theorem~\ref{thm:The_Big_Kahuna} will use Schmidt's outline of proof (see section \ref{sec:Siegel-Rogers}). It is a useful exercise to consider this proof in the case of Schmidt with $q = 3$. That is where $N_3 = \Z\setminus \{0\}$ and the constant $c(3) = \frac{\pi^2}{6} = \zeta(2)$. In section~\ref{sec:NumEv} we will see how the formula in Theorem~\ref{thm:The_Big_Kahuna} allows us to understand the asymptotic densities of saddle connections of translation surfaces with Veech group $H_q$ for $q$ odd. Theorem~\ref{thm:The_Big_Kahuna} is in fact a special case of the main theorem, which calculates the $k$th moment of the classical Siegel--Veech transform.
	
	\begin{thm}
	\label{thm: Higher Moments}
	Let $f\in B_c((\R^2)^k)$ and define 
	$$J_{n,m} =  \begin{bmatrix}
	1 & m \\ 0 & n
\end{bmatrix}.$$ Then
	\begin{align*}
		&\int_{Y_q} \hat{f} \,d\mu(g)\\
		 &= \sum_{\stackrel{\lambda \in \R^k}{\lambda = (1,\pm 1,\ldots,\pm 1)}} \frac{1}{c(q)} \int_{\R^2} f(\lambda x)\,dx \\
		&+ \sum_{n\in N_q} \sum_{\stackrel{0\leq m< \lambda_q|n|}{(m,n)^T\in V_q}} \sum_{1\leq j < k}\sum_{\lambda, \alpha, \beta} \frac{1}{c(q)} \int_{SL(2,\R)}f\left(\lambda g \begin{bmatrix} 1\\0 \end{bmatrix},g J_{n,m}\begin{bmatrix}
		\alpha \\ \beta
\end{bmatrix} \right)\,d\eta(g).\\
	\end{align*}
	where for each $1\leq j <  k$ we have $\lambda \in \R^j$ is of the form $(1,\pm 1,\ldots, \pm 1)$ and $\alpha = (0,\alpha_2,\ldots, \alpha_{k-j})$ and $\beta = (1,\beta_2, \ldots, \beta_{k-j})$ where for each $2\leq i\leq k-j$ we have
\begin{equation}\label{eq:alphabeta_Criterion}
\begin{bmatrix}
	\alpha_i\\ \beta_i 
\end{bmatrix} \in J_{n,m}^{-1} V_q.
\end{equation}
\end{thm} 
	\subsection{Outline}
		In Section~\ref{sec:Background and History} we give an overview of the history of the problem, followed by the necessary background on translation surfaces, Veech groups, and the geometric Euler totient function. In Section~\ref{sec:Orbits and integrals} we prove Theorem~\ref{thm:The_Big_Kahuna}, followed by Section~\ref{sec:Higher Moments} where we prove Theorem~\ref{thm: Higher Moments}. Finally in Section~\ref{sec:NumEv} we explain how we found numerical evidence for the result.
\subsection{Acknowledgments}
		I thank Jayadev Athreya for proposing the project and many useful discussions. I thank Bianca Viray and Claire Burrin for useful comments and discussions about the totient function. Thanks to Anthony Sanchez for useful comments on the paper. Finally I would like to thank Kimberly Bautista, Maddy Brown, and Andrew Lim of the Washington Experimental Mathematics Lab for their contributions to my numerical experiments and discussion in generalizing to higher moments.
\section{Background and history} \label{sec:Background and History}
	We first give a summary of previous related results in the geometry of numbers, followed by background on translation surfaces, Veech groups, and the $q$-geometric Euler totient function.
	
\subsection{Geometry of numbers}
\label{sec:Siegel-Rogers}
We will first focus on the mean and variance of the primitive Siegel transform, which is a special case of the Siegel--Veech transform defined in the previous section. First we set up some notation and definitions, then state the theorems of Siegel, Rogers, and Schmidt computing the mean and variance of the primitive Siegel transform.
	
	Consider $f\in B_c(\R^d)$. We aim to understand $f$ evaluated on \textit{visible} lattice points in $\R^d$, where a point $\vec{v} = (v_1,\ldots, v_d)^T\in\Z^d$ is \textit{primitive} or \textit{visible} if $\gcd(v_1,\ldots,v_d) = 1$. We denote the set of primitive vector points by $\Z^d_{prim}$, which one can show $\Z^d_{prim} = SL(d,\Z)\cdot \begin{bmatrix}
		1 \\ 0
	\end{bmatrix}$. Define $X_d = SL(d,\R)/SL(d,\Z)$. By abuse of notation, for an equivalence class $g = [g] \in X_d$, we define the \textit{primitive Siegel transform} by $$\hat{f}(g) = \sum_{\vec{v} \in \Z^d_{prim}} f(gv).$$  
	
	In 1945, Siegel \cite{Siegel45}, sections 5-6 showed
	\begin{equation}
		\int_{X_d} \hat{f}(g) \,d\mu(g) = \frac{1}{\zeta(d)}\int_{\R^d} f(x) \,dx
		\label{eq:Siegels_Primitive_Integral_Theorem}
	\end{equation}
	where the standard Lebesgue measure on $\R^d$ is $dx$, $\zeta$ is the Riemann zeta function, and $\mu$ is probability Haar measure on $X_d$. 
	
	In order to understand higher moments of $\hat{f}$, we split into the cases where $d = 2$ and $d > 2$. We address the latter case first.
	
	For understanding higher moments of $\hat{f}$, C.~A.~Rogers' 1955 paper \cite{Rogers55}, Theorem 5 solved the case for $\hat{f}^k$ with $d > 2$ and $k< d$. For simplicity, we will only consider the case $k =2$  of Rogers' result. Recall for $f \in B_c(\R^d)$, and defining $h\in B_c(\R^2\times \R^2)$ by $h(x,y) = f(x)f(y)$ we have
	$$\sum_{v_1,v_2 \in \Z^d_{prim}} h(gv_1, gv_2) = \sum_{v_1,v_2 \in \Z^d_{prim}} f(gv_1) f(gv_2) =\left[ \sum_{v \in \Z^d_{prim}} f(gv)\right]^2 = (\hat{f})^2.$$
	Rogers showed that for $f \in B_c(\R^d)$, and $h(x,y) = f(x)f(y)$, the second moment of $f$ is given by 
	\begin{align} \label{eq:RogersFormula}& \int_{X_d} (\hat{f})^2(g)	\,d\mu(g) =\int_{X_d} \sum_{v_1,v_2 \in \Z^d_{prim}} h(gv_1,gv_2) \,d\mu(g) \nonumber \\
	&\hspace{.25in}= \frac{1}{\zeta(d)^2} \int_{\R^d\times \R^d} h(x,y) \,dx\,dy + \frac{1}{\zeta(d)} \int_{\R^d} \left[h(x,x) + h(x,-x) \right]\,dx.
	\end{align}
	For a modern proof of Equation~\ref{eq:RogersFormula}, see section 4 of \cite{A-M_Log_Laws_I}. 
	
	For $k \geq d >2$ and $f\in B_c(\R^d)$, the function $\hat{f}^k$ is not integrable (Proposition 7.1 of \cite{Kleinbock_Margulis}). However when $d = 2$ we have $\hat{f}$ is bounded on $X_2$, and thus $\hat{f}^k$ integrable for any $k\geq 1$. So we now exclusively study the case $d= 2$. Rogers had a mistake in his paper claiming Equation~\ref{eq:RogersFormula} held for $d = 2$, which we can see does not work by setting $h_0$ to be the characteristic function of the set given by $$\{(v_1,v_2): \max(|v_1| , |v_2| ) \leq R \, , \, \det(v_1v_2) \notin \Z\,\}.$$
	Applying Equation~\ref{eq:RogersFormula} to $h_0$, the left hand side of Equation~\ref{eq:RogersFormula} will be identically zero as for any $v_1,v_2 \in \Z^2_{prim}$
	 $$\det(gv_1\, gv_2) = \det(g) \det(v_1 v_2) \in \Z,$$ and the right hand side will be nonzero as the vectors with integer determinant are  a Lebesgue measure zero subset of $\R^2 \times \R^2$. 
	
	In correction to Rogers, Schmidt addressed the case where $d= 2$ (see \cite{Schmidt60} Section 6). 
\begin{thm}\label{thm:Schmidt}
		Let $f\in B_c(\R^2\times \R^2).$  
		Then 
		\begin{equation}
		\int_{Y_3} \hat{f} \,d\mu(g)= \sum_{n\in \Z\setminus\{0\}} \frac{\varphi(n)}{\zeta(2)}\int_{SL(2,\R)} f\left(g J_n\right) \,d\eta + \frac{1}{\zeta(2)} \int_{\R^2} f(x, -x) + f(x,x) \,dx
		\label{eq:Full_Schmidt_Formula}
		\end{equation}
		where $\varphi$ is the standard Euler totient function. 
	\end{thm}
	
	Note this formula does not look exactly like the formula in Schmidt \cite{Schmidt60} as we have a different normalization of the Haar measure $\eta$. Note also that Theorem~\ref{thm:Schmidt} is a special case of Theorem~\ref{thm:The_Big_Kahuna}.

\subsection{Translation surfaces}\label{sec:Translation_Surfaces}
	A \textit{translation surface} is a surface formed by taking a finite number of polygons in the plane and gluing opposite sides by translation, where surfaces are equivalent up to cutting and pasting of these polygons via translation. Equivalently a translation surface is a closed Riemann surface $X$ with a nonzero holomorphic $1$-form $\omega$. This section will focus on examples relevant to this paper. For more background see \cite{masur2006ergodic}, \cite{hubert2006introduction}, \cite{eskin2006counting}.

		Given $A\in SL(2,\R)$ and $(X,\omega)$ a translation surface, we produce a new translation surface $A\cdot (X,\omega)$, which is the surface with charts of $(X,\omega)$ composed with $A$ acting linearly on $\R^2$. The \textit{Veech group} is the stabilizer subgroup of this action 
		$$SL(X,\omega) \stackrel{def}{=} \{A \in SL(2,\R): A\cdot (X,\omega) = (X,\omega)\}.$$	 The Veech group is always discrete and in fact trivial for almost every translation surface \cite{Gutkin_Judge_Geometry}.

		\textit{Saddle connections} on a translation surface $(X,\omega)$ are geodesics which start and end at zeros of the 1-form $\omega$ on $X$. For each saddle connection $\gamma$, there is an associated \textit{holonomy vector} $v_{\gamma} = \int_\gamma \omega \in \R^2$ which records the length and direction of $\gamma$.
		
		\subsection{Hecke triangle groups as Veech groups} \label{sec:Hecke_Triangle_As_Veech}
		We will consider surfaces whose Veech group is given by $SL(X,\omega) = H_q$ for $q\geq 3$. When $q= 3$, $H_3 = SL(2,\Z)$ which is the Veech group for the square torus. In general given a translation surface $(X,\omega)$ where we glue two regular $(2n+1)$-gons and then identify opposite sides, Veech showed in \cite{Veech89} that $SL(X,\omega) = H_{2n+1}$. For even Hecke triangle groups, Bouw and M\"oller \cite{BouwMoeller10} followed by a constructive proof of Hooper \cite{Hooper12} were able to show that there exists a translation surface $(X,\omega)$ with $SL(X,\omega)$ conjugate to an index $2$ subgroup of  $H_{2n}$, but there is no translation surface with Veech group containing $H_{2n}$. 
		
		Notice the set of holonomy vectors for the square torus are 
		$$\Z^2_{prim}= SL(2,\Z)\cdot \begin{bmatrix} 1 \\ 0 \end{bmatrix} = H_3\cdot \begin{bmatrix} 1 \\ 0 \end{bmatrix}.$$ 
		The characterization is not as clean for other surfaces, but if $(X,\omega)$ is a translation surface with $SL(X,\omega)$ a lattice, then the set of holonomy vectors will always be given as a finite union of $SL(X,\omega)$-orbits \cite{Veech89}, 5th paragraph section 3. By studying the Siegel--Veech transform over $V_q$ we will be able to understand asymptotic density of saddle connections for a class of translation surfaces \cite{Veech98}.

		\subsection{Geometric Euler totient function} \label{sec:euler}
		Recall we define the \textit{$q$-geometric Euler totient function} by
		$$\varphi_q(b) = \#\left\{1\leq a\leq \lambda_q \abs{b}: \begin{bmatrix} a\\ b \end{bmatrix} \in V_q\right\},$$
		where $\varphi_3 = \varphi$ is the standard Euler totient function.
		Since $V_q$ is discrete and thus $\varphi_q$ is finite and well defined.
		Though $\varphi_q$ generalizes the standard Euler totient function, $\varphi_q$ does \textit{not} agree with the more standard Euler totient function defined for the ring of integers over a number field in terms of the product formula over prime ideals. 
		
		 Following \cite{Lang_Lang}, we can define a \textit{greatest common $q$-divisor} denoted $(a,b)_q$ for $a,b \in \Z[\lambda_q]$ using a Euclidean pseudo-algorithm. This greatest common $q$-divisor has many similar properties to the gcd function, including for any $t\neq 0$, 
		\begin{equation}
			\label{eq:gcdmult}
			(ta, tb)_q = t \cdot (a,b)_q.
		\end{equation}
		With this definition we also have the following useful characterization of elements of $H_q$ as proved in Proposition~3.7 of \cite{Lang_Lang}.
		\begin{prop} \label{prop:H_q_Characterization}
	A matrix $\begin{bmatrix}
		a &  c \\ b & d
	\end{bmatrix} \in SL(2, \Z[\lambda_q])$ is in $H_q$ if and only if $(a,b)_q = (c,d)_q = 1$. In fact if $(a,b)_q < 1$ or $|b| < 1$, then $(a,b)^T$ cannot be a column of a matrix in $H_q$.
\end{prop}
	
\section{Orbits and integrals} \label{sec:Orbits and integrals}
	The goal of this section is to prove Theorem \ref{thm:The_Big_Kahuna}. 
	
	Let $f \in B_c(\R^2\times \R^2)$, and define $\hat{f}$ as in Definition \ref{def:S--VTransform}. Consider the map
	$$f\mapsto \int_{Y_q} \hat{f} d\mu.$$
	This mapping is a positive linear functional which is $SL(2,\R)$- invariant, where $SL(2,\R)$ acts diagonally by $g \cdot (v_1, v_2) = (gv_1,gv_2)$ for $(v_1,v_2) \in \R^2 \times \R^2$. Hence by the Riesz representation theorem, there exists a measure $\nu$ so that 
	$$\int_{\R^2\times \R^2} f\,d\nu = \int_{Y_q} \hat{f} \,d\mu \text{ for any }f \in B_c(\R^2\times \R^2).$$
	Since $\nu$ is $SL(2,\R)$-invariant, we can write $\nu$ as a combination of measures on $SL(2,\R)$ orbits of $\R^2\times \R^2$. So to understand $\nu$ we need to understand our integral over $SL(2,\R)$ orbits. 
	
	The outline of the proof is as follows. In section~\ref{sec:decomp_orbits} we split $\R^2\times \R^2$ into $SL(2,\R)$ orbits under the diagonal action and find the possible $SL(2,\R)$-invariant measures on these subsets. In section~\ref{sec:pf_reduction} we will reduce the uncountable number of orbits which occur in our setting to two linearly dependent orbits, and a countable number of linearly independent orbits. After setting up notation in section~\ref{sec:Notation}, in section~\ref{sec:pf_dependent} we reduce the linearly dependent case to Theorem~\ref{thm:Veech}, finally addressing the linearly independent case in section~\ref{sec:pf_independent}.
	
\subsection{Decomposition into orbits} \label{sec:decomp_orbits}
Let $\R^n_0 = \R^n\setminus\{\vec{0}\}$, similarly $\Z^n_0 = \Z^n \setminus \{{0}\}$, and $\Z_0[\lambda_q] = \Z[\lambda_q] \setminus \{0\}$.
	\begin{lemma}\label{lem:Decomposition_Into_Orbits}
		The following decomposes $\R^2\times \R^2$ into disjoint $SL(2,\R)$ orbits:
		$$\R^2 \times \R^2 =\left(\bigsqcup_{n\in \R_0} D_n \right) \sqcup \left(\bigsqcup_{t\in \R_0} LD_t\right) \sqcup H \sqcup V  \sqcup\{{ 0}\},$$	
		where we have the linearly independent determinants, 
		$$D_n = \{(\vec{v},\vec{w})\in \R^2 \times \R^2: \det(\vec{v}\; \vec{w}) = n\},$$
		the linearly dependent subsets
		$$LD_t = \{(\vec{v},t\vec{v}): \vec{v}\in \R^2_0\},$$
		and two special cases of linearly dependent vectors: horizontal and vertical 
		$$H = \{(\vec{v},0): \vec{v}\in \R^2_0\} \quad V = \{(0,\vec{v}): \vec{v}\in \R^2_0\}.$$
	\end{lemma}
	\begin{proof}
		We will realize each subset as an orbit of $SL(2,\R)$ under the diagonal action on $\R^2 \times \R^2$. 		
		Since $g\cdot \{{0}\} = 0$ for all $g\in SL(2,\R)$, the point $\{{0}\}$ is an entire orbit.
		
		Now notice that $g\cdot \begin{bmatrix} 1 \\ 0 \end{bmatrix} = \R^2_0$. Using this fact, for any $t\in \R_0$,
		$$SL(2,\R) \cdot \left(\begin{bmatrix}
			1 \\0
		\end{bmatrix}, \begin{bmatrix} t \\ 0 \end{bmatrix}\right) = \{(\vec{v}, t\vec{v}): \vec{v}\in \R^2_0\} = LD_t.$$
		
		Similarly for $H$ and $V$, it suffices to see that they are both given by 
		$$H = SL(2,\R) \cdot \left( \begin{bmatrix} 1 \\ 0\end{bmatrix}, \begin{bmatrix} 0 \\ 0\end{bmatrix}\right) \quad V = SL(2,\R) \cdot \left(  \begin{bmatrix} 0 \\ 0\end{bmatrix}, \begin{bmatrix} 1 \\ 0\end{bmatrix}\right).$$
		
		Finally, for $n\neq 0$, since $n\R^2_0 = \R^2_0$, 
		$$SL(2,\R) \cdot \left(\begin{bmatrix} 1 \\ 0 \end{bmatrix}, \begin{bmatrix}
			0\\ n
		\end{bmatrix}\right) = \{(\vec{v}, n\vec{u}): \det(\vec{v}\,\vec{u}) = 1\} = D_n.$$
		
		Thus we have shown each of these subsets is an $SL(2,\R)$ orbit. Finally, since every pair of elements in $\R^2$ is either linearly independent and thus have a nonzero determinant or linearly dependent and thus are scalar multiples we conclude every element $(v_1,v_2) \in \R^2 \times \R^2$ is contained in one of the given sets. Thus we have a decomposition of $\R^2\times \R^2$ into $SL(2,\R)$ orbits.
	\end{proof}

	The last task of this subsection is to determine the possible measures on each of our subsets. We will freely use the fact that Haar measure is unique up to scaling. In this section we will fix a particular scaling of Haar measure for each measure, and then by taking a linear combination of these different measures we can obtain $\nu$.
	
	On $\{0\}$, there is only one probability measure given by $\delta_0$, which is trivially $SL(2,\R)$ invariant.
	
On $H, V$, and $LD_t$ for $t\in \R_0$, we have a copy of $\R^2_0$. Notice Lebesgue measure $m_2$ on $\R^2$ is $SL(2,\R)$-invariant. So we will fix the standard Lebesgue measure giving the unit square $[0,1]^2$ volume $1$ on each of the subsets $H$, $V$, and $LD_t$ for $t\in \R_0$. Since $\{(0,0)\}$ is a measure zero subset, without loss of generality we can write integrals with respect to $m_2$ over all of $\R^2$. To see this measure is the unique $SL(2,\R)$-invariant measure (up to scaling), consider the induced Haar measure under the quotient of $SL(2,\R)/N \cong \R^2_0$ where $N = \left\{\begin{bmatrix}
	1 & t \\ 0 & 1
\end{bmatrix} : t\in \R\right\}$. 	
	
	To find a Haar measure on $D_n$, we will first find a Haar measure on $SL(2,\R)$, then we will show how this can be viewed as a Haar measure on $D_n$. To construct a Haar measure on $SL(2,\R)$, consider $SL(2,\R)$ as a subset of $(\R^4, m_4)$ where $m_k$ is Lebesgue measure on $\R^k$. As a result, for measurable $A\subseteq SL(2,\R)$, we can define the cone measure
	$$\eta(A) = m_4(C(A)) \text{ where } C(A) = \{\alpha g: \alpha \in (0,1], g\in A\}.$$
	Under matrix multiplication, $m_4$ is $SL(2,\R)$ invariant. Hence $\eta$ is an $SL(2,\R)$ invariant measure on $SL(2,\R)$. Under this measure, the set of matrices with a zero in the top left corner is a null set. Thus we can write the measure $d\eta = da\,db\,ds$ under the coordinates
	$$\begin{bmatrix}
	1 & 0 \\ s &1
	\end{bmatrix} \cdot \begin{bmatrix}
		a & b \\ 0 & a^{-1}
	\end{bmatrix}.$$
	With this normalization, in the quotient by $SL(2,\Z)$, we can compute the pushforward defined in terms of the projection map $\pi$ and fundamental domain $F$  \cite{Athreya_Cheung} $$(\eta)_*(Y_3) = \eta(\pi^{-1}(Y_3) \cap F) = \frac{\pi^2}{6} = \zeta(2).$$ 
	With this fixed normalization, $\eta$ gives the Poincar\' e volume. This means that we in fact have 
	$$(\eta)_*(Y_q) = c(q).$$
	
	Now having fixed Haar measure on $D_1$, for $D_n$ with $n\neq 0$, we identify $D_n$ with $D_1 = SL(2,\R)$ as $D_n = D_1 J_n$. Since we can write $D_n = D_1 J_n$, we choose the coordinates on $D_n$ to be the same as those on $D_1$. In this manner, we have $\eta$ is the Haar measure we will choose as our normalization of Haar measure on $D_n$.

	We've now decomposed $\R^2\times \R^2$ into $SL(2,\R)$ orbits, and fixed a normalization of Haar measure on each of these orbits. 
	
Since Haar measure is unique up to scaling, we can now write our $SL(2,\R)$ invariant measure on $\R^4$ as 
$$\nu = a \delta_0 + \sum_{t\in \R\cup \{\infty\}} b_t \, m_2  + \sum_{n\in \R_0} c_n \eta $$
for some constants $a,b_t, c_n$. where $b_\infty$ corresponds to $V$ and $b_0$ corresponds to $H$. 
\subsection{Reduction to visible determinants and removal of zero term}\label{sec:pf_reduction}
We have shown
\begin{align}\label{eq:Reduction_to_Integers}
\int_{Y_q} \hat{f} \, d\mu &= \int_{\R^2  \times \R^2} f \, d\nu\\
&= af(0,0) + \int_{t\in \R \cup \{\infty\}} b_t \int_{\R^2} f(x,tx) \,dx + \int_{n\in \R_0} c_n \int_{SL(2,\R)} f(g J_n) \,d\eta, \tag{where we define $(x,\infty x) = (0,x)$.}
\end{align}

The purpose of this section is to prove the following.
\begin{lemma}
	In Equation~(\ref{eq:Reduction_to_Integers}), $a = 0$, $t\in \{\pm 1\}$, and $n \in N_q$. 
\end{lemma} 

\begin{proof} To see that $t\in \{\pm 1\}$, consider the function $f$ supported on $LD_t$ for \newline $t\in \R \cup\{\infty\}$ where $LD_0 = H$ and $LD_\infty = V$. That is, for some large $R$ and $B(0,R)$ denoting the Euclidean ball in $\R^4$, let 
$$f_{R,t}(x,y) = \chi_{B(0,R)\setminus\{0\}}(x,y) \chi_{LD_t}(x,y) \quad \text{for } x,y \in \R^2.$$
On the left hand side of Equation \ref{eq:Reduction_to_Integers}, notice 
\begin{align*}
\hat{f}_{R,t}([g]) &= \sum_{v_1, v_2 \in V_q} f_{R,t}(gv_1,gv_2)\\
& = \#\{v_1,v_2 \in V_q \cap B(0,R): gv_1 = tgv_2\} \\
& = \#\{v_1,v_2 \in V_q \cap B(0,R): v_1 = tv_2\} .
\end{align*}
If $\begin{bmatrix}
	a	\\ c
\end{bmatrix} \in V_q$ by Proposition \ref{prop:H_q_Characterization}, we have $(a,c)_q = 1$, and thus by Equation (\ref{eq:gcdmult}), $(ta,tc)_q =t$. So by Proposition \ref{prop:H_q_Characterization} $\begin{bmatrix} ta \\ tc\end{bmatrix}$ cannot be an element of $V_q$ unless $t = \pm 1$. Or more geometrically since $V_q$ are the set of vectors visible from the origin, $tv$ is never visible from the origin unless $t = \pm 1$. Hence we've shown
$$\hat{f}_{R,t} =  \begin{cases} 0 & t\neq \pm 1\\\#\{v\in V_q \cap B(0,R)\} & t = \pm 1\end{cases}.$$
On the right hand side of Equation (\ref{eq:Reduction_to_Integers}), the only nonzero term will be the coefficient of $b_t$ for if $x\in B\left(0,\frac{R}{t}\right)$, then $tx \in B(0,R)$. Thus $\int_{\R^2} f_{R,t}(x,tx) \,dx \geq m_2\left(B_{\R^2}(0,\frac{R}{t})\right) >0$.  But when $t\neq \pm 1$, the left hand side of Equation~\ref{eq:Reduction_to_Integers} is zero since $\hat{f}_R = 0$. Hence $b_t = 0$ for $t \neq \pm 1$.

We now want to show that the set of possible determinants is $N_q$. For the determinant $n$ loci ($n\neq 0$), we similarly define 
$$f_{R,n}(x,y) = \chi_{B(0,R)}(x,y) \chi_{D_n}(x,y) \quad \text{for } x,y \in \R^2.$$
We compute
\begin{align*}
	\hat{f}_{R,n}([g]) &= \sum_{v_1,v_1 \in V_q} f_{R,n}(gv_1,gv_1)\\
	&= \#\{v_1,v_2 \in V_q \cap B(0,R): \det(v_1\,v_2)= \det(gv_1,gv_2) = n\}.\\
\end{align*}
	Since $N_q$ is the set of determinants that can arise as the determinant of two elements in $V_q$, we can write
	$$\hat{f}_{R,n} = \begin{cases} \#\{v_1,v_2 \in V_q \cap B(0,R): \det(v_1v_2) = n\} & n\in N_q\\ 0 & n\notin N_q. \end{cases}$$
	On the right hand side of Equation (\ref{eq:Reduction_to_Integers}), the only nonzero term corresponds to $c_n$, and $$\int_{SL(2,\R)} f_R\left(gJ_n	\right) \,d\eta >0$$ since  $D_n \cap B(0,R)$ has positive cone measure. In order to match the left hand side of Equation (\ref{eq:Reduction_to_Integers}) for $\hat{f}_{R,n}$, we conclude $c_n = 0$ for all $n\notin N_q$. 


We conclude this proof by showing $a = 0$. To see this, consider the characteristic function over the set $\{(0,0)\} \subseteq \R^2 \times \R^2$. That is set $f_0(x,y) = \chi_{\{(0,0)\}}(x,y)$. Then on the right hand side of Equation (\ref{eq:Reduction_to_Integers}), we have $f_0(0,0) = 1$, all other integrals are zero since $\{(0,0)\}$ is a measure zero subset of $\R^2$, and cannot show up in $SL(2,\R) J_n$ for any $n$. Thus the right hand side of Equation (\ref{eq:Reduction_to_Integers}) for $f_0$ is $a$. On the left hand side of Equation (\ref{eq:Reduction_to_Integers}), $(0,0)$ is not a pair of visible vectors since $(0,0)$ cannot be the first column of a matrix in $H_q$, so the left hand side is zero. Thus we conclude $a= 0$.
\end{proof}	
	
	To summarize, in this section we reduced our Equation (\ref{eq:Reduction_to_Integers}) to
	\begin{cor}\label{cor:Reduced_to_nonzero_coeffiecients}
		$$\int_{Y_q} \hat{f} \,d\mu = \int_{\R^2} b_1 f(x,x) + b_{-1} f(x,-x) \,dx + \sum_{n\in N_q} c_n \int_{SL(2,\R)} f\left(gJ_n\right)\,d\eta.$$
	\end{cor}
	
	\subsection{Notation and division into smaller lemmas}\label{sec:Notation}
	In the proceeding sections, we will compute the values for $b_1, b_{-1}$, and $c_n$ for $n\in N_q$. In order to do this, we introduce the following notation:
	for $D$ a discrete subset of $(\R^2)^k$ which is $V_q$-invariant under the diagonal action, define $f_{D}: Y_q \to \R$ by 
		$$f_D([g]) = \sum_{v\in D^k} f(gv)$$
	 In a similar manner define the functional $T_D: B_c((\R^{2})^k) \to \R$ by 
	 $$T_D(f) = \int_{Y_q} f_D([g]) \, d\mu([g]).$$ 
	 We now define the following sets:
	$$D_n^{V} = D_n \cap (V_q \times V_q) = \{(v,w)\in V_q \times V_q : \det(v\, w) = n\},$$
	$$LD_{\pm 1}^{V} = \{(v, \pm v) : v\in V_q\}.$$
	 
	Then we can rewrite the left hand side of Corollary~\ref{cor:Reduced_to_nonzero_coeffiecients} as
	\begin{align*}&\int_{Y_q}\left[ f_{LD_{1}^{V}} + f_{LD_{-1}^{V}} + \sum_{n\in N_q} f_{D_n^{V}}\right] \,d\mu \\
	&= T_{LD_{1}^{V}}(f) + T_{LD_{-1}^{V}}(f) + \sum_{n\in N_q} T_{D_n^{V}}(f).
	\end{align*}
Thus finding the coefficients in Corollary~\ref{cor:Reduced_to_nonzero_coeffiecients} is reduced to finding coefficients individually in each of these equations:
	\begin{equation}\label{eq:LD_case_coefficients}
		T_{LD_{\pm 1}}(f) = b_{\pm 1} \int_{\R^2} f(x,\pm x)\,dx,
	\end{equation}
	and for each $n\in N_q$
	\begin{equation}\label{eq:D_n_case_coefficients}
		T_{D_n^{V}}(f) = c_n \int_{SL(2,\R)} f\left(g J_n \right)\,d\eta.
	\end{equation}
\subsection{Reducing to Siegel--Veech formula in linearly dependent case} \label{sec:pf_dependent}
	In this section, we will prove that the coefficients $b_1$ and $b_{-1}$ in Equation (\ref{eq:LD_case_coefficients}) are given by $b_1 = b_{-1} = \frac{1}{c(q)}$ by reducing to the Siegel--Veech Primitive Integral Formula (Theorem \ref{thm:Veech}). That is, we will prove the following:
	
	\begin{lemma} \label{lem:Siegel_Linearly_Dependent_Case}
		For any $f \in B_c(\R^2\times \R^2)$, 
		$$ T_{LD_{\pm 1}^{V}}(f) = \frac{1}{c(q)} \int_{\R^2} f(v,\pm v) \,dv$$
		where $c(q)$ is the Poincar\' e volume of the unit tangent bundle over $\mathbb{H}^2 / H_q$.
	\end{lemma}
 	\begin{proof}
 		Given $f \in B_c(\R^2\times \R^2)$, define $\bar{f} \in B_c(\R^2)$ by
 		$$\bar{f}_{\pm}(u) = f(u, \pm u).$$
 		So we now compute
 		\begin{align*}
 			T_{LD_{\pm 1}^{V}} (f) &= \int_{Y_q} \sum_{v\in V_q} f(gu, \pm gu) \,d\mu([g]) \\
 			&= \int_{Y_q} \sum_{u\in V_q} \bar{f}_{\pm}(gu)\,d\mu([g]) \\
 			&= \frac{1}{c(q)} \int_{\R^2} \bar{f}_{\pm}(x)\,dx \tag{by Theorem \ref{thm:Veech}}\\
 			&= \frac{1}{c(q)} \int_{\R^2} f(x,\pm x)\,dx.
 		\end{align*}
 		This concludes the proof of the lemma.
 	\end{proof}

We've now shown $b_{\pm 1} = c(q)^{-1}$, in the next section, we address the coefficients $c_n$ for $n\in N_q$.
\subsection{Coefficients on loci with fixed determinant} \label{sec:pf_independent}
 The goal of this section is to prove that each $c_n = c(q)^{-1} \varphi_q(n)$ for $n \in N_q$. We will first decompose $D_n^{V}$ into $H_q$ orbits under the diagonal action, showing there are $\varphi_q(n)$ orbits which each contribute equally to $T_{D_n^{V}}$. After showing this, we will find the value over a single orbit.
 
\begin{lemma} \label{lem:N_q equality}
	Let $n \in \Z_0[\lambda_q]$. There exists $v_1 , v_2 \in V_q$ with $\det(v_1 v_2) = n$ if and only if there exists $m\in \Z[\lambda_q]$ with $0\leq m< \lambda_q|n|$ and $\begin{bmatrix}
		m \\ n 
	\end{bmatrix} \in V_q$.
	
	In particular, the equality in Equation (\ref{eq:N_q}) for $N_q$ holds. 
\end{lemma}
\begin{proof}
	First, suppose there exists $m \in \Z[\lambda_q]$ with $0\leq m< \lambda_q|n|$ and $\begin{bmatrix} m \\ n \end{bmatrix} \in V_q$. Set $v_1 = \begin{bmatrix}
		1 \\ 0
\end{bmatrix}$, which is in $V_q$ since $T v_1 = v_1$, and set $v_2 = \begin{bmatrix}
	m \\ n
\end{bmatrix}$. Then $v_1, v_2 \in V_q$ with determinant $n$, so $n \in N_q$. 

Conversely suppose $v_1, v_2 \in V_q$ with $\det(v_1v_2) = n$. Let $g\in H_q$ so that $ge_1 = v_1$. Then,
$$g^{-1} [v_1 \, v_2] = [e_1\, g^{-1}v_2].$$
Since the determinant is $n$,
$$g^{-1}v_2 = \begin{bmatrix}
	\ell \\
	n
\end{bmatrix}$$ for some $\ell \in \Z[\lambda_q]$. Applying the matrix $T$ to the left $j$ times for some $j\in \Z$ gives
$$T^{j}g^{-1} [v_1\,v_2] = [e_1 \, T^jg^{-1} v_2]$$ where
$$T^j g^{-1}v_2 = \begin{bmatrix}
	\ell + j n \lambda_q \\ n
\end{bmatrix}.$$

Thus we can find $j \in \Z$ so that $m = \ell + jn\lambda_q$ satisfies 
$$0 \leq m < \lambda_q |n|,$$
and the proof is complete.
\end{proof}
\begin{lemma} \label{lem:Decompose_Into_SL_2_Z_Orbits}
 For $n \in N_q$	The subset $D_n^{V}$ is the union of $\varphi_q(n)$ different orbits
 	$$D_n^{V} = \bigsqcup_{\stackrel{0\leq m < |n|\lambda_q}{(m,n)^T\in V_q}} E_n^{(m)},$$
	where 
	$$E_n^{(m)} = \left\{ \gamma \cdot \begin{bmatrix}
		1 & m \\ 0 & n
\end{bmatrix}	 : \gamma \in H_q\right\}.$$
\end{lemma}
\begin{proof}
We will first show that the decomposition of every element in $D_n^{
V}$ can be written as an element $E_n^{(m)}$ for some $m$. 
	
	Let $$\begin{bmatrix}
		a & b \\ c & d
\end{bmatrix}	  \in D_n^{V}$$
	
	From the proof of Lemma \ref{lem:N_q equality}, there exists a matrix $h \in H_q$ with
	$$h \begin{bmatrix} a & b \\ c & d\end{bmatrix} = \begin{bmatrix} 1 & m \\ 0 & n \end{bmatrix} \text{ for } 0 \leq m < \lambda_q |n|.$$
	Since $\begin{bmatrix}
	b \\ d
\end{bmatrix} \in V_q$, we also have $\begin{bmatrix} m \\ n\end{bmatrix} = h\begin{bmatrix}
	b \\ d
\end{bmatrix} \in V_q$. We have now shown every element in $D_n^{V}$ is in $E_n^{(m)}$ for some $m$ with $0\leq m < \lambda_q|n|$ and $\begin{bmatrix} m \\ n\end{bmatrix} \in V_q$.

To see that we have no duplicate representatives of our orbits, let $0< m_1, m_2 \leq \lambda_q|n|$ with $m_1\neq m_2$ and $\begin{bmatrix} m_1 \\ n\end{bmatrix}, \begin{bmatrix} m_2 \\ n\end{bmatrix}\in V_q$. Without loss of generality suppose $m_2 > m_1$. If the representatives
$$\begin{bmatrix}
	1 & m_1 \\ 0 & n
\end{bmatrix} \quad \text{and}\quad\begin{bmatrix}
	1 & m_2 \\ 0 &n
\end{bmatrix}$$
were in the same $H_q$ orbit, there would exist an element $\begin{bmatrix}
	a & b \\ c & d
\end{bmatrix} \in H_q$ such that 
$$\begin{bmatrix}
	a & b\\ c & d
\end{bmatrix} \begin{bmatrix}
	1 & m_1 \\ 0 & n
\end{bmatrix}  = \begin{bmatrix}
	1 & m_2 \\ 0 & n
\end{bmatrix}.$$ 
This implies $a = 1, c= 0, d=1,$ and $b = \frac{m_2-m_1}{n}.$ Since $0< m_2 - m_1 < \lambda_q n$, we have $0< b < \lambda_q$. This is a parabolic element with upper right entry smaller than the generating matrix $T$, and thus not in $H_q$. Therefore we conclude that $E_n^{(m)}$ are $\varphi_q(n)$ distinct $H_q$ orbits whose union is all of $D_n^{V}$.
\end{proof}

\begin{lemma} \label{lem:Fibration_On_One_Orbit}
	For a fixed $m$ with $0\leq m < \lambda_q|n|$ and $\begin{bmatrix} m \\ n\end{bmatrix}\in V_q$, 
	$$T_{E_n^{(m)}}(f) =\frac{1}{c(q)} \int_{SL(2,\R)} f\left(gJ_n \right)\,d\eta_2.$$ 
\end{lemma}
\begin{proof}
	Let $\pi: SL(2,\R) \to Y_q$ be the projection map $g\mapsto [g]$. Recall we normalize $\eta$ so that $\pi_*(\eta) (Y_q) =c(q)$. Hence $\pi_*\left(c(q)^{-1}\eta\right) = \mu$. Moreover, to push a function from $SL(2,\R)$ to a function  on $X_2$, we have to sum over the orbits $H_q$. Thus, 
	$$\frac{1}{c(q)} \int_{SL(2,\R)} f\left(g\begin{bmatrix}
		1 & m \\ 0 & n
	\end{bmatrix} \right)\,d\eta = \int_{Y_q} \sum_{\gamma \in H_q}f\left(g \cdot \gamma\cdot \begin{bmatrix}
		1 & m \\ 0 & n
	\end{bmatrix} \right) \,d\mu = T_{E_n^{(m)}}(f).$$ 
	
	For the last part of the lemma, we compute the following
	\begin{align*}
		&\frac{1}{c(q)} \int_{SL(2,\R)} f\left(g\begin{bmatrix}
		1 & m \\ 0 & n
	\end{bmatrix} \right)\,d\eta  \\
		&= \frac{1}{c(q)} \int_{SL(2,\R)} f\left(g\begin{bmatrix}
		1 & \frac{1-m}{n}\\ 0 &1\end{bmatrix} \begin{bmatrix} 1 & m \\ 0 & n \end{bmatrix}\right)\,d\eta\left(g\begin{bmatrix}
		1 & \frac{1-m}{n}\\ 0 &1
	\end{bmatrix}\right) \\
	&= \frac{1}{c(q)} \int_{SL(2,\R)} f\left(g\begin{bmatrix}
		1 & 1\\ 0 &n
	\end{bmatrix} \right)\,d\eta(g) 
	\end{align*}
	where the last equality follows from the fact that $SL(2,\R)$ is a unimodular group, so the Haar measure $\eta$ is both left and right invariant under the action of $SL(2,\R)$. 
\end{proof}

Lemma \ref{lem:Fibration_On_One_Orbit} shows that $T_{E_n^{(m)}}(f)$ is constant for with respect to $m$. Hence we conclude
$$T_{D_n^{V}} = \sum_{\stackrel{0\leq m<\lambda_q |n|}{(m,n)^T\in V_q}} T_{E_n^{(m)}} = \frac{\varphi_q(n)}{c(q)}  \int_{SL(2,\R)} f\left(gJ_n \right)\,d\eta(g).$$

In conclusion, we've now shown that 
$$T_{D_n^{V}} = \frac{\varphi_q(n)}{c(q)} \int_{SL(2,\R)} f\left(gJ_n\right)\,d\eta\left(g\right)$$
As well as 
$$T_{LD_{\pm 1}} = \frac{1}{c(q)} \int_{\R^2} f(x,\pm x) \,dx.$$

Putting these results together with Corollary~\ref{cor:Reduced_to_nonzero_coeffiecients}, we have now shown Theorem~\ref{thm:The_Big_Kahuna} holds.

\section{Higher moments}\label{sec:Higher Moments}
We will prove Theorem~\ref{thm: Higher Moments} which is the generalization of Theorem~\ref{thm:The_Big_Kahuna} which corresponds to higher moments of the classical Siegel--Veech Transform on $\R^2$. 

\subsection{Decomposition into orbits}
	We first decompose $(\R^2)^k$ into $SL(2,\R)$ orbits. Given a point in $(\R^2)^k$, either all the terms are linearly dependent, or there exist two terms in the $k$-tuples which are linearly dependent.
		
	\begin{lemma}\label{lem: Higher moment decomp into oribts}
		The following decomposes $(\R^2)^k$ into disjoint $SL(2,\R)$ orbits:
		
		$$(\R^2)^k =  \left( \bigsqcup_{\lambda} LD_{\lambda}\right) \sqcup \left(\bigsqcup_{n, \lambda, \alpha, \beta} D_{n,\lambda,\alpha,\beta}\right).$$
		In the linearly dependent case, $$LD_{\lambda} = \{\lambda x: x\in \R^2\}$$  for $\lambda \in \R^k$ with first nonzero entry (if it exists) given by 1. 
		
		In the linearly independent case,
		$$D_{n,\lambda, \alpha, \beta} \stackrel{def}{=} \{(\lambda x, \alpha x + \beta y) : \det(x\, y) = n\},$$
		where $n \in \R_0$ is the determinant of the first nonzero vector with the first linearly independent vector. For $0\leq j < k$ we have $\lambda \in \R^j$ where the first nonzero entry is 1 and $\alpha, \beta \in \R^{k-j}$ where $\alpha = (0, \alpha_2,\ldots, \alpha_{k-j})$ and $\beta = (1, \beta_2,\ldots, \beta_{k-j})$.
	\end{lemma}

		\begin{proof}
			We first claim that $LD_{\lambda}$ and $D_{n,\lambda, \alpha, \beta}$ can be written as $SL(2,\R)$ orbits. Indeed since $SL(2,\R)$ acts transitively and linearly by matrix multiplication on $\R^2\setminus\{0\}$ we can write
			$$SL(2,\R) \cdot \lambda \begin{bmatrix}
			1 \\ 0
			\end{bmatrix} = LD_{\lambda}.$$
			Similarly since $SL(2,\R)$ acts transitively on determinant $n$ subsets as proved in Lemma \ref{lem:Decomposition_Into_Orbits} and linearly on $\R^2$, we can write
			$$SL(2,\R) \cdot \left(\lambda \begin{bmatrix}
				1 \\0
			\end{bmatrix}, \alpha \begin{bmatrix} 1 \\ 0\end{bmatrix} + \beta \begin{bmatrix}
				0 \\ n
\end{bmatrix}			 \right) = D_{n,\lambda,\alpha,\beta}.$$

			Next we show that the union of the orbits in fact covers all of $(\R^2)^k$. To do this consider a vector $v = (v_1,v_2,\ldots, v_k) \in (\R^2)^k$. If all $v_i = 0$ then $v \in LD_{\vec{0}}$. Otherwise there is some first nonzero $v_i$ which we will call $x$. If $\dim(\spn(v_1,\ldots, v_k)) = 1$, then every other element will be a linear multiple of $x$. Hence $v \in LD_{\lambda}$ where $\lambda$ has first nonzero entry is $1$ and all remaining entries are real numbers.
			
			If however $\dim(\spn(v_1,\ldots, v_k)) = 2$, then set $y$ to be the first vector after $x$ which is linearly independent of $x$. For all $v_i$ which occur after $y$, $v_i$ can be written as a linear combination of $x$ and $y$, thus written as $\alpha_i x + \beta_i y$ for some real numbers. Since $LD_{\lambda}$ and $D_{n,\lambda,\alpha,\beta}$ are subsets of $(\R^2)^k$ we conclude that 
			$$(\R^2)^k = \left(\bigcup_{\lambda} LD_{\lambda}\right) \cup \left(\bigcup_{n,\lambda,\alpha, \beta} D_{n,\lambda,\alpha,\beta}\right).$$
			
						Finally we finish the proof by proving each of these orbits is distinct. Since all pairs of entries in $LD_{\lambda}$ have determinant $0$ and $SL(2,\R)$ preserves determinants, we know that the $LD_{\lambda}$ and $D_{n,\lambda,\alpha,\beta}$ must be disjoint. 
						
	Now suppose that $LD_{\lambda} = LD_{\lambda'}$. Since the first nonzero vector must have a coefficient of $1$, if $\lambda_j = 1$ is the first nonzero element, then $\lambda_j' = 1$ as well. Now every vector after the first vector is linearly dependent on the first nonzero vector, so there is a unique coefficient and $\lambda = \lambda'$.
	
	Similarly, since we choose a coefficient of $1$ for the first nonzero vector, and a coefficient of $1$ for the first vector which is linearly independent, we have a unique representation of the linear combinations $\alpha_i x + \beta_i y$. Hence the $D_{n,\lambda, \alpha, \beta}$ are also all disjoint. This completes the proof.
		\end{proof}
	
\subsection{Reduction to smaller lemmas}

Using the notation of section~\ref{sec:Notation} we will rewrite $\int_{Y_q} \hat{f} \,d\mu$, reducing the proof of Theorem~\ref{thm: Higher Moments} to smaller lemmas.

Define $D_{n,\lambda,\alpha,\beta}^V = D_{n,\lambda,\alpha, \beta} \cap (V_q)^k$ and $LD_{\lambda}^V = LD_{\lambda} \cap (V_q)^k$. Moreover for $1\leq  m\leq |n|$ with $(m,n)^T \in V_q$ define 
$$E_{n,\lambda,\alpha, \beta}^{(m)} = H_q \cdot \left( \lambda \begin{bmatrix} 1\\ 0 \end{bmatrix}, J_{m,n} \begin{bmatrix}
	\alpha \\ \beta 
\end{bmatrix}\right)$$
where $\begin{bmatrix}
	\alpha \\ \beta
\end{bmatrix}$ is a $2\times (k-j)$ matrix and
$$J_{n,m} =  \begin{bmatrix}
	1 & m \\ 0 & n
\end{bmatrix}$$
corresponds to $J_n = J_{n,1}$ in the $k=2$ case.
By Lemma~\ref{lem:Decompose_Into_SL_2_Z_Orbits}, we can write 
$$D_{n,\lambda,\alpha,\beta} = \bigsqcup_{\stackrel{0\leq m< \lambda_q |n|}{(m,n)^T\in V_q}} E_{n,\lambda,\alpha, \beta}^{(m)}.$$

\begin{lemma}\label{lem:Higher moment decomp}
We have
$$\int_{Y_q} \hat{f} \,d\mu = \sum_{\lambda} T_{LD_{\lambda}^V} + \sum_{n\in N_q} \sum_{\stackrel{0\leq m< \lambda_q |n|}{(m,n)^T \in V_q}} \sum_{\lambda,\alpha, \beta} T_{E_{n,\lambda,\alpha, \beta}^{(m)}}$$
where in the linearly dependent case $\lambda \in \R^k$, where the first element of $\lambda$ must be 1, and any remaining elements of $\lambda$ must be $\pm 1$.

In the linearly independent case, given $n\in N_q$, there exists a unique $0\leq m < \lambda_q|n|$ so that the two vectors lie in the $H_q$ orbit of $J_{m,n}$ 
Given $n$ and $m$, we have $\lambda \in \R^j$ with first entry $1$ and all remaining elements $\pm 1$. Moreover $\alpha = (0, \alpha_2, \ldots, \alpha_{k-j})$ and $\beta = (1, \beta_2, \ldots, \beta_{k-j})$ where $\alpha$ and $\beta$ satisfy Equation~\ref{eq:alphabeta_Criterion} for each $2\leq i\leq k-j$.
\end{lemma}
\begin{proof}
	By Lemma~\ref{lem: Higher moment decomp into oribts}, given $v \in (V_q)^k$, we must have $v \in LD_{\lambda}^V$ or $D_{n,\lambda,\alpha,\beta}^V$ for some $\lambda$ or $(n,\lambda,\alpha,\beta)$. 
	
	If $v \in LD_{\lambda}^V$, first note that the zero vector is not in $V_q$. Hence the first entry in $\lambda$ must be $1$. Since the vectors are in $V_q$ and must be constant multiples of the first vector, all other vectors must be $\pm v_1$. Thus $\lambda$ must have the specified form.  
	
	Moving onto the linearly independent case, if $v \in D_{n,\lambda,\alpha, \beta}^V$, then we can write $v = (\lambda v_1, \alpha v_1 + \beta v_j)$ for some $1<j\leq k$ where $v_j$ is the first vector after $v_1$ which is not co-linear with $v_1$. Since all the vectors in $\lambda v_1 $ must be in $V_q$, $\lambda$ must have first entry $1$ and all other entries $\pm 1$. 
	
	Setting $\det(v_1\, v_j) = n$ by the definition of $N_q$ we have $n \in N_q$.
	
	Finally we need to determine the criterion for $\alpha = (0,\alpha_2,\ldots, \alpha_{k-j})$ and $\beta = (1,\beta_2, \ldots, \beta_{k-j})$. So that for all $1\leq i\leq k-j$ we have $\alpha_i v_1 + \beta_i v_j  \in V_q$. 
	
	By Lemma~\ref{lem:Decompose_Into_SL_2_Z_Orbits}, 
	there exists $\gamma \in H_q$ and $0\leq m < \lambda_q|n|$ with $(m,n)^T \in V_q$ so that $$\gamma \cdot (v_1|v_j) = \begin{bmatrix}
		1 & m \\ 0 & n
	\end{bmatrix}.$$
	So we have
	$$\gamma \cdot (\alpha_i v_1 + \beta_i v_j) = \begin{bmatrix}
		\alpha_i + m \beta_i \\ n\beta_i \end{bmatrix}.$$
	Since $H_q$ acts transitively on $V_q$, $\alpha_i v_1 + \beta_i v_j \in V_q$ is equivalent to
	$$\begin{bmatrix}
		1 & m \\ 0 & n
	\end{bmatrix} \begin{bmatrix}		\alpha_i \\ \beta_i \end{bmatrix} = \begin{bmatrix}
	\alpha_i + m\beta_i \\ n \beta_i
\end{bmatrix} \in V_q.$$
Multiplying by the inverse of $J_{n,m}$ we see $\alpha$ and $\beta$ must satisfy equation~\ref{eq:alphabeta_Criterion}.
	
	So for each $n$, there exists $m$ with $0\leq m< \lambda_q |n|$ so that $(m,n)^T \in V_q$. Given this $m$, we already have the requirement for $\lambda$, and the $\alpha$ and $\beta$ must satisfy Equation~\ref{eq:alphabeta_Criterion}. This concludes the proof. 
\end{proof}

Now that we have decomposed $\int_{Y_q} \hat{f} \,d\mu$, the higher moments case is complete once we prove the following two lemmas.

\begin{lemma}\label{lem:Higher Moment lin. dependent}
Given the restrictions of $\lambda$ in Lemma~\ref{lem:Higher moment decomp},
	$$T_{LD_\lambda^V} = \frac{1}{c(q)} \int_{\R^2} f(\lambda v) \,dv.$$
\end{lemma}
\begin{proof}
	The proof strategy is identical to the strategy in Lemma~\ref{lem:Siegel_Linearly_Dependent_Case}
\end{proof}

\begin{lemma}\label{lem:Higher moment lin.independent}
Given the restrictions of $n,m,\lambda,\alpha, \beta$ in Lemma~\ref{lem:Higher moment decomp}
	$$T_{E_{n,\lambda,\alpha,\beta}^{(m)}} = \frac{1}{c(q)} \int_{SL(2,\R)}f\left(\lambda g \begin{bmatrix} 1\\0\end{bmatrix}, g J_{n,m}\begin{bmatrix}
		\alpha \\ \beta
\end{bmatrix}		   \right)\,d\eta(g).$$

\end{lemma}
\begin{proof}
	This proof strategy is identical to the proof in Lemma~\ref{lem:Fibration_On_One_Orbit}. Note we cannot use the change of variables to get equal contribution for each $T_{E_n^{(m)}}$ as in the case of Lemma~\ref{lem:Fibration_On_One_Orbit} because the criterion for which $\alpha$ and $\beta$ can occur depends on $m$. 
\end{proof}

Combining Lemma~\ref{lem:Higher moment decomp} with Lemma~\ref{lem:Higher Moment lin. dependent} and Lemma~\ref{lem:Higher moment lin.independent}, we have now concluded the proof of Theorem~\ref{thm: Higher Moments}.

\section{Numerical evidence}	\label{sec:NumEv}
This section discusses how to interpret Theorem~$\ref{thm: Higher Moments}$ in terms of a counting problem. We will focus on the case $k =2$, that is Theorem~\ref{thm:The_Big_Kahuna}. The following proposition is from section 16 of \cite{Veech98}.
\begin{prop}
	For $B(0,R)$ the ball of radius $R$ in $\R^2$, 
	$$\lim_{R\to\infty} \frac{\#\{V_q \cap B(0,R)\}}{\pi R^2}  = \frac{1}{c(q)}.$$
\end{prop}
We will use the notation $\#\{V_q \cap B(0,R)\} \sim \frac{\pi R^2}{c(q)}$ where $f(R) \sim g(R)$ if and only if $\lim_{R\to\infty} f(R)/g(R) = 1$.

In the case $q= 3$ this can be interpreted as the probability a randomly chosen integer vector is primitive is $\frac{1}{\zeta(2)}$. See Figure~\ref{fig:Densities_Integers} and Figure \ref{fig:Densities} for visualization of the points of $V_q$ for $q=3,4,5$. 

To construct the set $V_q$, we used a Farey tree construction in the first quadrant and then used the 4-fold symmetry of $V_q$, that is $(a,b) \in V_q$ implies $(a,-b), (-a,b), (-a,-b) \in V_q$. The generalization of the Farey tree construction as found in \cite{Lang_Lang} begins with the vectors 
$$\begin{bmatrix}
	1 \\ 0
\end{bmatrix} = \begin{bmatrix}
	1 & 0 \\ 0 & 1
\end{bmatrix}\cdot \begin{bmatrix}
1 \\ 0
\end{bmatrix} \in V_q \text{ and } \begin{bmatrix} 0 \\ 1 \end{bmatrix} = \begin{bmatrix}
	0 & -1 \\ 1 & 0
\end{bmatrix}\cdot \begin{bmatrix}
	1 \\ 0
\end{bmatrix} \in V_q.$$
 Then for $i = 2,\ldots, q-1$ we add the vectors
$$\begin{bmatrix}
	a_i \\ b_i
\end{bmatrix} = \begin{bmatrix}\lambda_q a_{i-1} - a_{i-2} \\  \lambda_q b_{i-1} - b_{i-2} \end{bmatrix}$$
where 
 $$\begin{bmatrix} a_1 \\  b_1 \end{bmatrix} = \begin{bmatrix} 1\\ 0 \end{bmatrix} \text{ and } \begin{bmatrix} a_0\\ b_0\end{bmatrix} = \begin{bmatrix}0 \\ -1\end{bmatrix}.$$ 
Iterating this step between each pair of adjacent vectors we obtain the elements of $V_q$.

\begin{figure}
	\centering
	\includegraphics[scale=.3]{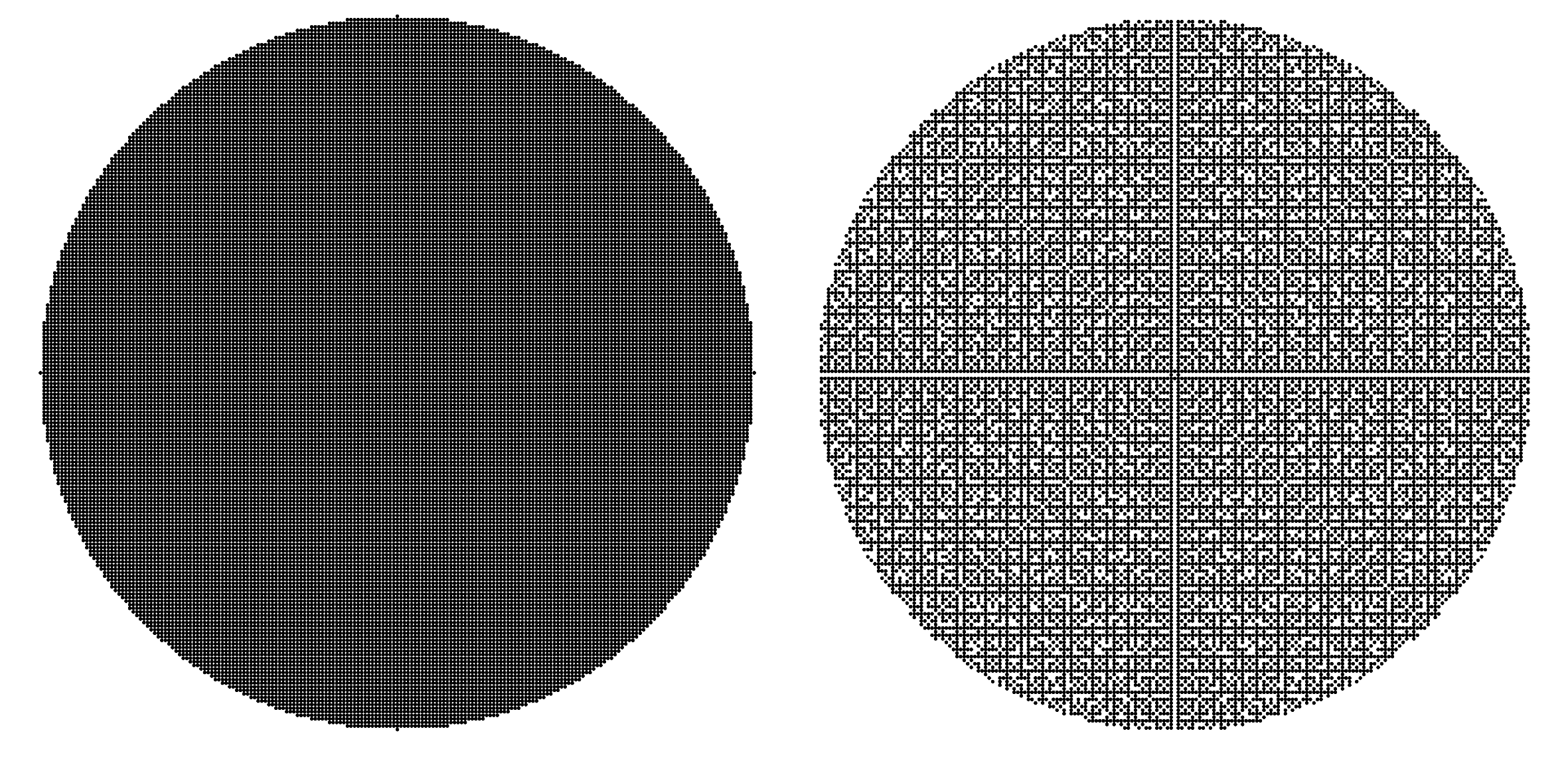}
	\caption{For $R = 100$, the left hand plot includes all integer pairs $(a,b)$ and the right hand plot includes all primitive integers pairs $(a,b)$ with $gcd(a,b) = 1$ where $a^2+b^2 \leq 100^2$. This demonstrated the expectations that for large $R$, the number of points on the left should be approximately $\pi 100^2$, but on the right the number of points should be $\frac{\pi 100^2}{\zeta(2)}$.}
	\label{fig:Densities_Integers}
\end{figure}
\begin{figure}
	\centering
	\includegraphics[scale=.8]{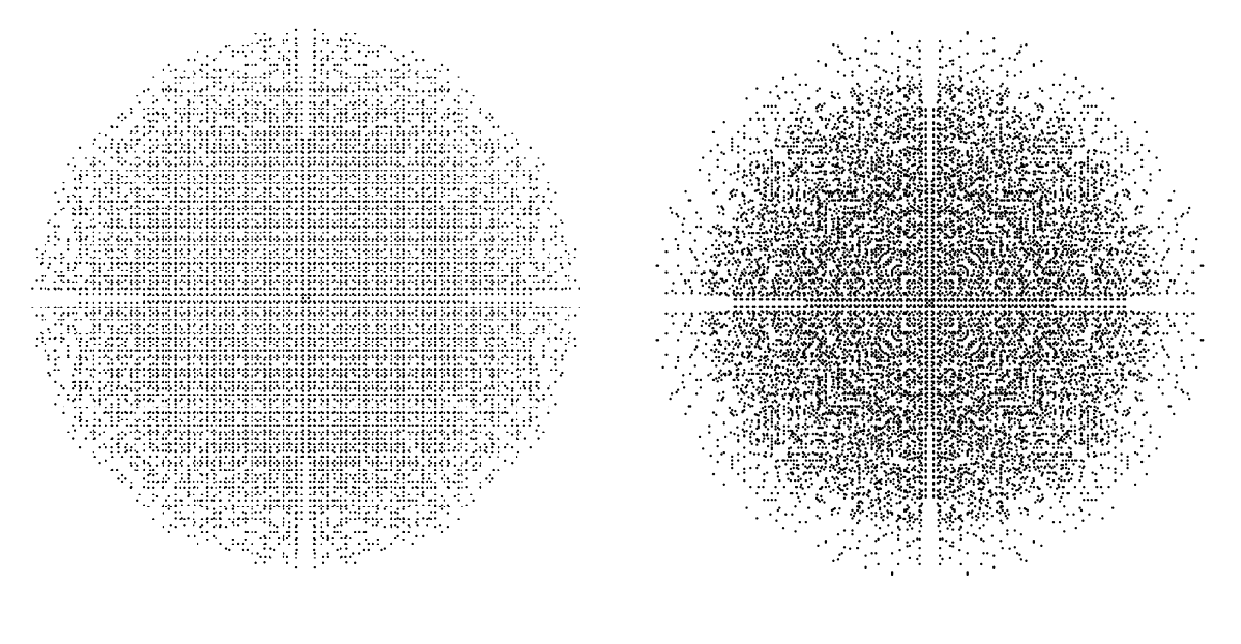} 
	\caption{On the left is a plot of the vectors $V_4$, and on the right is a plot of the vectors $V_5$. These plots were generated using the Farey Tree construction.}
	\label{fig:Densities}
\end{figure}

Now that we've generated plots for $V_q$, we can now count pairs of elements in $V_q$ corresponding to the square of the Siegel--Veech transform. Specifically for $f = \chi_{B_{\R^4}(0,R)}$ the characteristic function of the Euclidean ball in $\R^4$, we want to understand $T_{D_n^{V}}(f)$ which will asymptotically grow like the function
$$\cnt_q(R,n) \stackrel{def}{=}\#\left\{ \begin{bmatrix}
	a & b \\ c & d
\end{bmatrix} \in D_n^V : \, a^2 + b^2 + c^2 + d^2 \leq R^2\right\}.$$

Theorem~\ref{thm:The_Big_Kahuna} states that 
$$\frac{\cnt_q(R,n)}{\int_{SL(2,\R)} f\left(g J_n\right) \,d\eta}\sim \frac{\varphi_q(n)}{c(q)},$$
which is not useful for understanding data without more knowledge about $\int_{SL(2,\R)} f(gJ_n)\,d\eta$. 

Newman \cite{Newman} showed that $\cnt_3(R,1) \sim 6 R^2$. In particular combining with Theorem~\ref{thm:The_Big_Kahuna}, we obtain
$$\int_{SL(2,\R)} f(gJ_1) \, d\eta \sim \pi^2 R^2.$$

Next using the result of Schmidt \cite{Schmidt60}, we can extend this result to the fact that when $q= 3$,
$$\int_{SL(2,\R)} f\left(g J_n\right) \,d\eta \sim \frac{\pi^2}{n}R^2 .$$

Thus we deduce that for any $q \geq 3$,
$$\frac{\cnt_q(R,n)}{R^2} \sim \frac{1}{c(q)} \cdot \varphi_q(n) \cdot \frac{\pi^2}{n} = \frac{\varphi_q(n)\cdot  \pi^2}{n \cdot c(q)}.$$ 

Indeed in our numerical experiments we obtained the desired results. In Figure \ref{fig:Count}, we show the convergence for $k = 1,2,3,4$. Recall $D_n^V$ can be decomposed into $\varphi_q(n)$ orbits $E_n^{(m)}$ where $0\leq m < \lambda_q|n|$, and on each orbit we were able to verify we had density asymptotic to $\frac{\pi^2}{n \cdot c(q)}$ as desired. Finally in Figure~\ref{fig:pairs}, we provide a visualization for pairs of elements in $V_q$ for $q=3$ and $q=5$.

\begin{figure}
	\centering
	\includegraphics[scale=1]{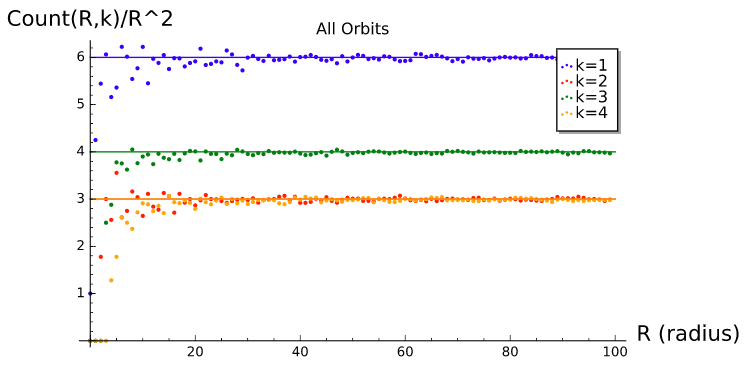}
	\caption{For $q= 3$ this is a plot of $R$ on the horizontal axis, and a plot of $\frac{\cnt(R,k)}{R^2}$ on the vertical axis. Notice that $k=2$ and $k= 4$ both converge to $3$ since $\frac{6\varphi(4)}{4} = \frac{6 \varphi(2)}{2} = 3$.}
	\label{fig:Count}
\end{figure}
\begin{figure}
	\centering
	\includegraphics[scale=1]{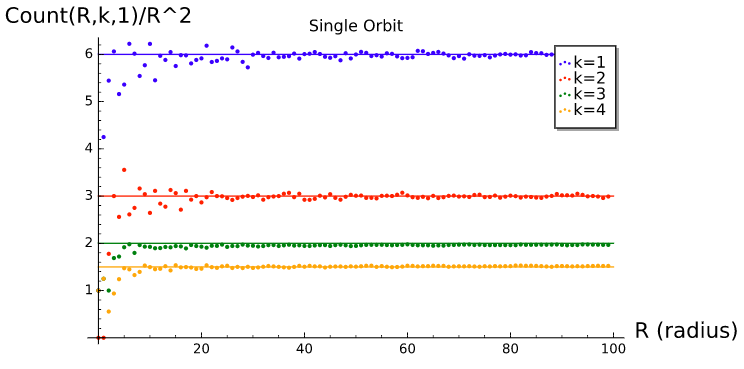}
	\caption[]{For $q= 3$, a plot of $R$ on the horizontal axis, and a plot of $\frac{\cnt(R,n,1)}{R^2}$ on the vertical axis, where $\cnt(R,n,1)$ is number of elements within the ball of radius $R$, which are in the orbit $SL(2,\Z)\cdot J_n$.}
	\label{fig:CountOrbit}
\end{figure}

\begin{figure}
	\centering
	\includegraphics[scale=.6]{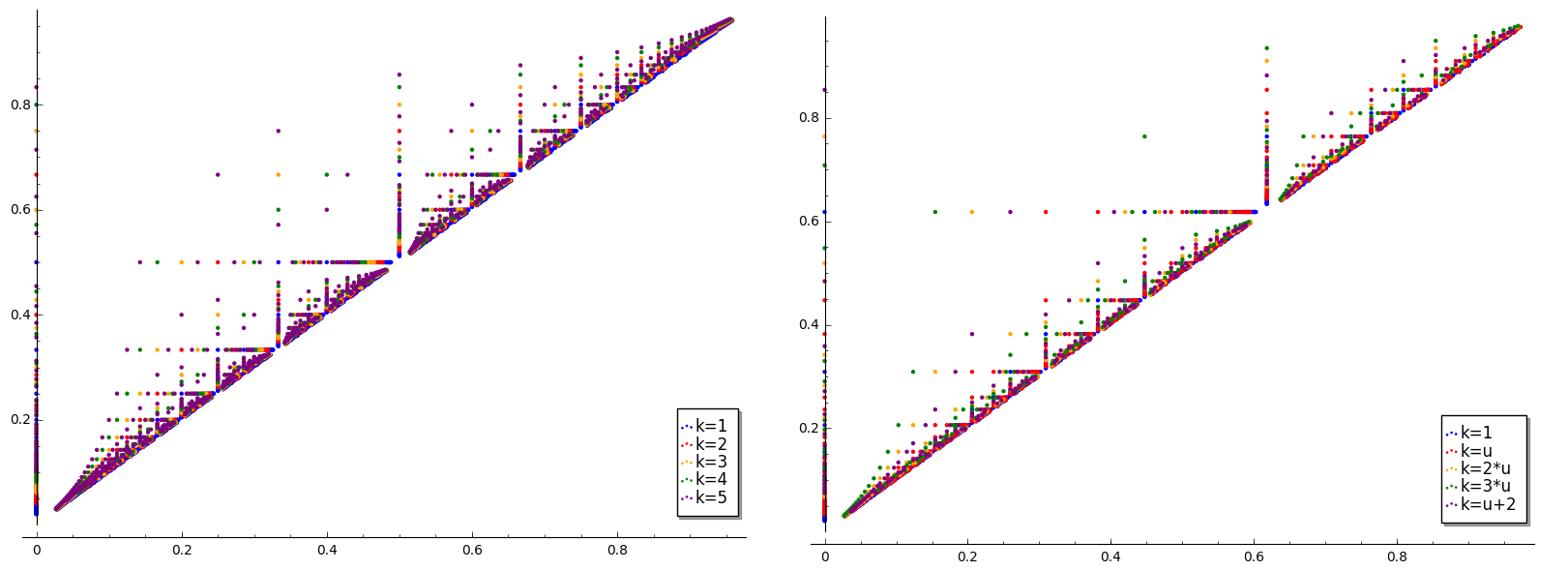}
	\captionsetup{singlelinecheck=false}
	\caption{For each element $(a,b)$ in $V_q$, we can visualize these points by considering $a/b$, which for simplicity we will consider the pairs with $a/b \in [0,1]$. These pictures plot the points in $D_k^V$ by placing a point if the point $a/b$ on the $x$-axis and the point $c/d$ on the $y$-axis have the corresponding pairs $(a,b)$ and $(b,c)$ in $D_k^V$. On the left are pairs of primitive vectors, and on the right are pairs of vectors in $V_5$ where in the legend $u = \phi $, the golden ratio.}
	\label{fig:pairs}
\end{figure}

\bibliographystyle{alpha}
\bibliography{Sources}
\end{document}